\documentclass[12pt,a4paper]{article}

\usepackage{amsmath,amssymb,amsthm}
\usepackage[shortlabels]{enumitem}
\setlist[enumerate]{label=(\roman*)}
\usepackage{graphicx}
\usepackage{tikz}

\usepackage[font=small]{caption}
\usepackage[colorlinks=false]{hyperref}
\hypersetup{
    colorlinks,
    linkcolor={red!50!black},
    citecolor={blue!50!black},
    urlcolor={blue!80!black}
}
\usepackage[capitalise]{cleveref}
\usepackage{xcolor}
\usepackage{calrsfs}
\usepackage{natbib}
\usepackage{subcaption}

\allowdisplaybreaks

\newtheorem{theorem}{Theorem}
\newtheorem{lemma}[theorem]{Lemma}
\newtheorem{prop}[theorem]{Proposition}
\newtheorem{cor}[theorem]{Corollary}
\theoremstyle{definition}

\newtheorem{remark}[theorem]{Remark}

\newcommand \dm[1]  { \,\mathrm d{#1} }

\renewcommand{\epsilon}{\varepsilon}
\renewcommand{\phi}{\varphi}

\renewcommand{\leq}{\leqslant}
\renewcommand{\geq}{\geqslant}

\newcommand{\norm}[1]{\left\lVert#1\right\rVert}

\newcommand{\norml}[3]{\norm{#1}_{L^{#2}({#3})}}

\newcommand{\abs}[1]{\left\vert#1\right\vert}

\renewcommand{\hat}[1]{\widehat{#1}}

\newcommand{\e}{\mathrm{e}}
\renewcommand{\i}{\mathrm{i}}

\newcommand{\tpi}{2 \pi \mathrm{i}}

\newcommand{\R}{\mathbb{R}}

\newcommand{\N}{\mathbb{N}}

\newcommand{\Z}{\mathbb{Z}}

\newcommand{\T}{\mathbb{T}}

\renewcommand{\th}{\textsuperscript{th} }
\newcommand{\qtext}[1]{\quad\text{#1}\quad}
\newcommand{\qand}{\qtext{and}}
\newcommand{\qfa}{\qtext{for all}}
\newcommand{\ie}{i.e.\ }
\newcommand{\eg}{e.g.}

\title{The $L^p$ convergence of Fourier series on triangular domains}
\author{Ryan L. Acosta Babb%
\thanks{University of Warwick, UK
(\href{mailto:r.acosta-babb@warwick.ac.uk}{r.acosta-babb@warwick.ac.uk}).}}
\date{}

\begin{document}
\maketitle

\begin{abstract}
    We prove $L^p$ norm convergence for (appropriate truncations of) the Fourier series
    arising from the Dirichlet Laplacian eigenfunctions on three types of triangular domains in $\R^2$:
    (i) the $45$-$90$-$45$ triangle,
    (ii) the equilateral triangle and
    (iii) the hemiequilateral triangle (\ie half an equilateral triangle cut along its height).
    The limitations of our argument to these three types are discussed in light of Lam\'e's Theorem.
\end{abstract}

\section{Introduction}\label{sec:intro}

In one dimension, there is only one way to truncate the partial sums of a Fourier series \[
    \sum_{n=-\infty}^\infty \hat{f}(n)\e^{\tpi nx},
\]namely \[
    S_N(x) := \sum_{n=-N}^N \hat{f}(n)\e^{\tpi nx}.
\] Then the Fourier series converges if, and only if, $S_N(x)$ converges as $N\to\infty$.

When moving to higher dimensions, we have \[
    \sum_{m=-\infty}^\infty \sum_{n=-\infty}^\infty \hat{f}(m,n)\e^{\tpi (mx+ny)}
\]and an ambiguity arises.

Since the eigenvalue for $\e^{\tpi(mx+ny)}$ is proportional to $m^2+n^2$,
the ``natural choice'' of truncation for the partial sums is
\[
    \sum_{m^2+n^2\leq N^2}\hat{f}(m,n)\e^{2\pi\i (mx+ny)};
\]that is, we cut off the sum once we have picked out all eigenfunctions with eigenvalues
$\abs{\lambda_{m,n}}\lesssim N^2$.
(Geometrically, this procedure corresponds to using a ``circular cutoff'' in frequency space,
by choosing frequencies $(m,n)$ in the ball of radius $N$.)
A celebrated result of \cite{Fefferman1971} implies,
with the help of standard transference results \citep[Chap.\,4]{GrafakosCFA},
that such ``eigenvalue truncations" of Fourier expansions
fail, in general, to converge to $f$ in the $L^p(\T^2)$ norm when $p\neq 2$.

We may, instead, truncate according to the labelling of the indices (or frequencies) $(m,n)\in\Z^2$: \[
    \sum_{\abs{m},\abs{n}\leq N}\hat{f}(m,n)\e^{\tpi (mx+ny)}
\] Happily, these ``truncations by label" always converge back to $f$ in \emph{all} $L^p(\T^2)$ spaces
(provided that $1<p<\infty$).
\citep[See][for a proof of this classical result.]{GrafakosCFA}
Recently, \cite{RobinsonFefferman2021} have noted that the general problem of
finding ``well-behaved'' truncations of eigenfunction expansions in $L^p$-based spaces is still open
for more general bounded Euclidean domains.

A natural starting point is to consider the next ``simplest'' domains, such as discs or triangles.
The eigenfunctions of the disc are products of trigonometric and Bessel functions
and so share a similar product structure to the classical Fourier series on $\T^2$.
Specifically, the eigenfunctions are of the form \[
    \e^{\tpi \theta n}J_{\abs{n}}(j_m^{\abs{n}}r)
\] giving rise to the multidimensional Bessel--Fourier series
\begin{equation}\label{eqn:bfseries}
    \sum_{n=-\infty}^{\infty}\sum_{m=1}^\infty a_{m,n}\e^{\tpi n\theta}J_{\abs{n}}(j_m^{\abs{n}}r).
\end{equation}
(Here, $J_{\abs{n}}$ denotes a Bessel function of the first kind and $j_m^{\abs{n}}$ its non-negative zeros.)
The best result we were able to find in the literature is due to \cite{CordobaBalodis},
who proved norm convergence in the mixed norm space $L^p_{\mathrm{rad}}( L^2_{\mathrm{ang}} )$
defined by the condition\[
    \norm{f}_{p,2} := \left[\int_0^1\left(\sum_{k}\abs{f_k(r)}^2\right)^{p/2}r\dm{r}\right]^{1/p} <\infty,
\]where $f_k(r)$ is the $k$\th Fourier coefficient of the angular function $f(r,\cdot)$ for fixed $r$: \[
    f_k(r) := \int_{0}^{1}f(r,\theta)\e^{\tpi k\theta}\dm{\theta}.
\] By truncating the series \eqref{eqn:bfseries} in the ranges $\abs{n}\leq N$, $m\leq M$,
they were able to show that there is a constant $A>0$ such that the Bessel--Fourier series
of $f \in L^p_{\mathrm{rad}}( L^2_{\mathrm{ang}} )$ converges to $f$ in the $\norm{\cdot}_{p,2}$ norm
provided that $M\geq AN+1$ and $4/3 < p < 4$.
Furthermore, the endpoints for the range of $p$ are sharp.
\citep[See][Theorem 2, of which our discussion is a special case when $d=2$.]{CordobaBalodis}

By modifying their proof, we were able to improve the result to $L^p$ convergence with respect to the usual measure $r\dm{r}\dm{\theta}$
on the disc, provided that $2\leq p <4$ and \[
    \norm{f}_{p,q} := \left[\int_0^1\left(\sum_{k}\abs{f_k(r)}^q\right)^{p/q}r\dm{r}\right]^{1/p} <\infty,
    \qtext{where} \frac{1}{p} + \frac{1}{q} = 1.
\] As far as we know, the problem of $L^p$ convergence for functions $f\in L^p(r\dm{r}\dm{\theta})$
in the range $4/3<p<4$ is still open; see \cite{Acosta2022}.

We therefore turn our attention to triangular domains, which turn out to be much more amenable to analysis.
Following early work of \cite{Lame1833}, other authors such as \cite{Prager1998} and \cite{McCartin2003} have derived
explicit trigonometric expressions for Dirichlet eigenfunctions on the equilateral triangle.
Neither of these authors, however, consider questions of convergence for the associated eigenfunction series,
and subsequent work by \cite{Adcock2009}, \cite{Huybrechsetal2010}, and \cite{SunLi2005},
restricts attention to numerical methods or convergence in $L^2$-based spaces such as $H^k$.
\citep[See][for a comprehensive survey of the literature on the Laplacian and its eigenfunctions.]{Grebenkov2013}
To our knowledge, there is no treatment of the $L^p$ convergence of series of eigenfunctions on the triangle for $p\neq 2$.

In this paper we establish $L^p$ convergence for trigonometric series of eigenfunctions for the Dirichlet Laplacian
on three types of triangular domains:
(i) the $45$-$90$-$45$ triangle,
(ii) the equilateral triangle and
(iii) the hemiequilateral triangle (\ie half an equilateral triangle cut along its height).

Eigenfunctions for the later types, (ii) and (iii), were obtained in the abovementioned work of \cite{Prager1998} and \cite{McCartin2003}.
Using their insights, we prove that, on each of these domains, any $L^p$ function
can be written as a norm-convergent series of eigenfunctions.
We also give a complete treatment of these issues for the simpler case (i).
As far as we know, the result for type-(i) triangles is new.

In \cref{sec:isosceles} we derive the eigenfunctions for the type-(i) triangles,
prove their completeness and establish $L^p$ convergence of the associated series.
This serves as a useful prototype for the arguments we will develop in the following sections.

The result for equilateral domains is obtained by a decomposition of a function into
a symmetric and antisymmetric part with respect to the line $x=0$.
This reduces the problem to type-(iii) domains, which, once solved, easily yields the equilateral case.

Thus, in \cref{sec:assym,sec:symm} we separately derive the antisymmetric and symmetric modes,
following the ingenious approach of \cite{Prager1998} (see \cref{fig:t1arefl}).
The hemiequilateral triangle $T_1$ is tiled into the rectangle ${R=[0,\sqrt{3}]\times[0,1]}$,
where we exploit the classical $L^p$ convergence of ``double-sine'' and ``cosine-sine'' series on $R$
to derive convergence on $T_1$.
This is our main contribution, and is worked out in \cref{sec:assymconv,sec:symmconv}.
These results are combined in \cref{sec:results} to obtain $L^p$ convergence on the equilateral triangle.

In \cref{sec:conc} we conclude with some remarks about the limitations of this argument:
owing to a theorem of Lam\'e, the three domains (i)--(iii) listed above are the ``only'' ones amenable to this procedure.
\section{The 45-90-45 triangle}\label{sec:isosceles}

Taking our cue from Práger's analysis of the equilateral case,
we reflect the 45-90-45 triangle along the hypotenuse to obtain a square;
see \cref{fig:sqrefl}.

\begin{figure}[ht]
    \begin{center}
        \includegraphics{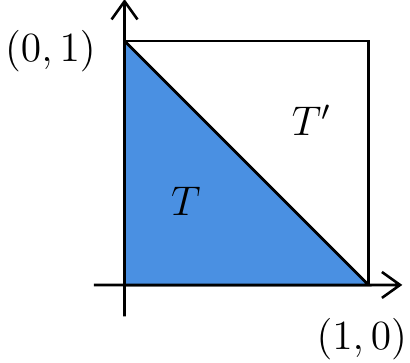}
        \caption{Reflecting $T$ into the square $[0,1]^2$ along the diagonal $y=1-x$.}
        \label{fig:sqrefl}
    \end{center}
\end{figure}

Placing the triangle, $T$, at vertices $(1,0)$, $(0,1)$ and $(0,0)$, the reflections have the simple from
\begin{equation*}
    x = 1 - \eta \qand y = 1-\xi \qtext{where} (\xi,\eta) \in T.
\end{equation*}
Given a function $f:T\to\R$ we define the \emph{prolongation} of $f$\[
    \mathcal{P}f(x,y) := \begin{cases}
        f(x,y) , &\text{if } (x,y)\in T;\\
        -f(1-y,1-x), &\text{if } (x,y)\in T';
    \end{cases}
\] where $T'$ denotes the reflected triangle, so that $[0,1]^2=T\cup T'$.

We now expand $\mathcal{P}f$ as a double-sine series on $[0,1]^2$:
\begin{align}\label{eqn:Pfcoef}
    \hat{\mathcal{P}f}(m,n) &:= 4\int_{0}^1\int_0^1 \mathcal{P}f(x,y)\sin(\pi mx)\sin(\pi ny)\dm{x}\dm{y}\nonumber\\
    &= 4\int_{T}f(\xi,\eta)\bigg[\sin(\pi m\xi)\sin(\pi n\eta)\nonumber\\
    &\qquad\quad {}-\sin(\pi m(1-\eta))\sin(\pi n(1-\xi))\bigg]\dm{\xi}\dm{\eta}\nonumber\\
    &= 4\int_{T}f(\xi,\eta)u_{m,n}(\xi,\eta)\dm{\xi}\dm{\eta},
\end{align}
where we defined the functions \[
    u_{m,n}(\xi,\eta) := \sin(\pi m\xi)\sin(\pi n\eta) - \sin(\pi m(1-\eta))\sin(\pi n(1-\xi))
\] for all points $(\xi,\eta)\in T$.
\begin{lemma}\label{lemma:umnsquare}
    We record the following facts about the functions $u_{m,n}$ (indexed by $m,n\in\N$).
    \begin{enumerate}
        \item For all $m\neq n$: \[
            u_{m,n} = \sin(\pi m \xi)\sin(\pi n\eta)-(-1)^{m+n}\sin(\pi n\xi)\sin(\pi m\eta),\]
            and clearly $u_{m,n}\equiv 0$ when $n=m$.
        \item Symmetry of the indices: \[
            u_{n,m} = -(-1)^{m+n}u_{m,n} \qfa m\neq n;
        \] therefore, it suffices to consider the set with $0 < m < n$.
        \item Letting $(x,y)$ range over $[0,1]^2$, \[
            \mathcal{P}\left[u_{m,n}|_T\right](x,y) = u_{m,n}(x,y) \qfa m\neq n.
        \]
        In other words: restricting $u_{m,n}$ from $[0,1]^2$ to $T$ and prolonging by $\mathcal{P}$ yields $u_{m,n}$ once more,
        as a function on $[0,1]^2$.
        \item The set $\{u_{m,n}:0<m<n\}$ is a complete, orthogonal set in $L^2(T)$ and \[
            \norml{u_{m,n}}{2}{T} = \frac{1}{2}.
        \]
    \end{enumerate}
\end{lemma}
\begin{proof}
    Statements (i)-(iii) are easily verified. For the completeness claim in (iv), suppose that $f\in L^2(T)$ is such that \[
        \int_T fu_{m,n} = 0 \qfa 0 < m < n.
    \] It follows from (i), (iii) and \eqref{eqn:Pfcoef} that the Fourier coefficients of $\mathcal{P}f$ on $[0,1]^2$ vanish,
    so $\mathcal{P}f=0$ in $L^2([0,1]^2)$, and therefore $f=\mathcal{P}f|_{T}=0$.
    This proves the completeness of the set.

    From (iii) and \eqref{eqn:Pfcoef}, it is clear that \[
        \int_T u_{m,n}u_{k,l}  = \int_0^1\int_0^1 u_{m,n}(x,y)\sin(\pi kx)\sin(\pi ly)\dm{x}\dm{y}.
    \] Ordering the indices as $0<m<n$ and $0<k<l$, it follows from the formula in (i) that
    the functions $u_{m,n}$ and $u_{k,l}$ are pairwise orthogonal; setting $m=k$ and $n=l$ easily yields
    the norm of $u_{m,n}$.
\end{proof}

\begin{cor}
    The function $u_{m,n}$ with $0<m<n$ is an eigenfunction of (minus) the Dirichlet Laplacian on $T$
    with eigenvalue $\pi^2(m^2+n^2)$.
\end{cor}
\begin{proof}
    It is immediate from (i) and (iii) in \cref{lemma:umnsquare} that \[
        -\Delta u_{m,n}(x,y) = \pi^2(m^2+n^2)u_{m,n}(x,y)
    \] pointwise for all $(x,y)\in [0,1]^2$, and $u_{m,n}\equiv 0$ on the boundary of the square.
    Thus, restricting to $T\subset [0,1]^2$ yields eigenfunctions of the Laplacian on $T$
    that clearly vanish along the edges parallel to the axes.
    It remains to show that $u_{m,n}$ also vanishes on the hypotenuse $y=1-x$,
    but this is immediately verified on substitution into the formula from (i) above.
\end{proof}

We thus have a complete, orthogonal set of eigenfunctions for the Dirichlet Laplacian on $T$.
We can define the Fourier coefficients as usual: \[
    f^{\triangle}(m,n) := \frac{1}{\norml{u_{m,n}}{2}{T}^2}\int_T fu_{m,n} = \hat{\mathcal{P}f}(m,n);
\] the last equality holds by \cref{eqn:Pfcoef}.

The partial sum operators, truncated by indices, for the Fourier series on $T$ and $[0,1]^2$ are
\begin{align*}
    S_N^Tf &:= \sum_{0<m<n\leq N} f^{\triangle}(m,n)u_{m,n}, \qtext{and}\\
    S_N^{[0,1]^2} &:= \sum_{0<m,n\leq N}\hat{\mathcal{P}f}(m,n)\sin(\pi mx)\sin(\pi ny),
\end{align*}respectively.

The next proposition tells us that the prolongation operation commutes with the Fourier partial sums,
allowing us to``push forward" the $L^p$ convergence results from the square to the triangle.

\begin{prop}\label{prop:PScommuteSquare}
    Let $f\in L^p(T)$. Then: \[
        \mathcal{P}\left[S_N^Tf\right](x,y) = S_N^{[0,1]^2}\left[\mathcal{P}f\right](x,y)
        \qfa (x,y)\in[0,1]^2. 
    \]
\end{prop}
\begin{proof}
    Break up the lattice $[1,N]^2\cap \N^2$ into the following regions:
    \[
        I: 0 < m < n \leq N \qand II: 0 < n < m \leq N.
    \] (Recall that, by (i) in \cref{lemma:umnsquare}, the diagonal $m=n$ yields no eigenfunctions.)
    Hence:
    \begin{align*}
        S_N^{[0,1]^2}\left[\mathcal{P}f\right](x,y) &=\sum_{I}\hat{\mathcal{P}f}(m,n)\sin(\pi mx)\sin(\pi ny)\\
        &\qquad \quad{}+ \sum_{II}\hat{\mathcal{P}f}(m,n)\sin(\pi mx)\sin(\pi ny)\\
        &=\sum_{I}\hat{\mathcal{P}f}(m,n)\bigg[\sin(\pi mx)\sin(\pi ny)\\
        &\qquad \quad{}-(-1)^{m+n}\sin(\pi nx)\sin(\pi my)\bigg] &\text{\cref{lemma:umnsquare} (ii)}\\
        &= \sum_{I}\hat{\mathcal{P}f}(m,n)u_{m,n}(x,y) &\text{\cref{lemma:umnsquare} (i)}\\
        &= \sum_{I}\hat{\mathcal{P}f}(m,n)\mathcal{P}u_{m,n}(x,y) &\text{\cref{lemma:umnsquare} (iii)}\\
        &= \mathcal{P}\left[S_N^Tf\right](x,y),
    \end{align*}
    where in the last line we make use of the linearity of $\mathcal{P}$ and the identity
    $f^{\triangle}(m,n)=\hat{\mathcal{P}f}(m,n)$ for all $(m,n)$ in $I$.
\end{proof}

We now have everything we need to prove our result.

\begin{theorem}
    Let $f\in L^p(T)$ with $1<p<\infty$. Then, $S_N^Tf\to f$ in $L^p(T)$.
\end{theorem}
\begin{proof}
    Note that, for $f\in L^p$, we have \[
        \norml{\mathcal{P}f}{p}{[0,1]^2}^p = 2\norml{f}{p}{T}^p.
    \] Thus, by successively applying this equality and \cref{prop:PScommuteSquare}, we have
    \begin{align*}
        \norml{S_N^Tf-f}{p}{T}^p &= \frac{1}{2}\norml{\mathcal{P}\left[S_N^Tf\right]-\mathcal{P}f}{p}{[0,1]^2}^p\\
        &= \frac{1}{2}\norml{S_N^{[0,1]^2}\mathcal{P}f-\mathcal{P}f}{p}{[0,1]^2}^p \to 0,
    \end{align*}
    by the $L^p$ convergence of Fourier double-sine series on $[0,1]^2$.
\end{proof}
\section{Antisymmetric eigenfunctions on the equilateral triangle}\label{sec:assym}

To tackle convergence on the equilateral triangle $T$, we divide it along its midline $x=0$
into two congruent hemiequilateral triangles, and call the right-hand one $T_1$
(see \cref{fig:t1arefl}).
Any function $f:T\to \R$ may be decomposed into a symmetric and antisymmetric part
with respect to $x$, \ie $f=f_a+f_s$ where \[
    f_a(x,y) := \frac{f(x,y)-f(-x,y)}{2} \qand
    f_s(x,y) := \frac{f(x,y)+f(-x,y)}{2}
\] for all $(x,y)\in T$.
Clearly each of $f_a$ and $f_s$ is determined by its values on $T_1\subset T$.

Let us begin with the $L^p$ convergence of the Fourier series of $f_a$ on the triangle $T_1$.
Denote by $(\xi,\eta)$ the coordinates of a point in $T_1$.
Then, for each coordinate pair $(x,y)$ in the rectangle $R=[0,\sqrt{3}]\times[0,1]$,
we write \cite[following][p.\,312]{Prager1998} \[
    (x,y) = (x_i(\xi,\eta),y_i(\xi,\eta)) \qtext{for exactly one} (x_i,y_i)\in T_i.
\] Note that the $i$ subscript serves just as a reminder that the point $(x_i,y_i)$
lies in the region $T_i$ of the Cartesian plane.

The \emph{antisymmetric prolongation} of $u\colon T_1\to\R$ to the rectangle $R$ is defined by
\[
    \mathcal{P}_a u(x,y) := \begin{cases}
        u(\xi,\eta) &\text{if } (x,y) = (x_1,y_1)\in T_1,\\
        -u(\xi,\eta) &\text{if } (x,y) = (x_2,y_2)\in T_2,\\
        u(\xi,\eta) &\text{if } (x,y) = (x_3,y_3)\in T_3,\\
        u(\xi,\eta) &\text{if } (x,y) = (x_4,y_4)\in T_4,\\
        -u(\xi,\eta) &\text{if } (x,y) = (x_5,y_5)\in T_5,\\
        u(\xi,\eta) &\text{if } (x,y) = (x_6,y_6)\in T_6.\\
    \end{cases}
\]

\begin{figure}[ht]
    \begin{center}
        \includegraphics{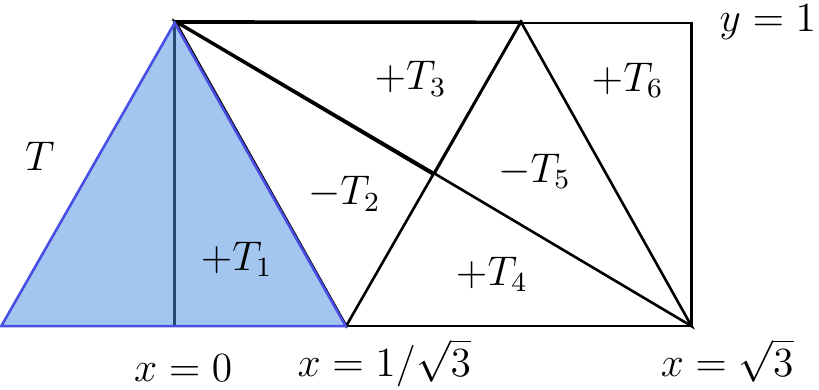}
        \caption{The equilateral triangle $T$ (shaded), the right-angled triangle $T_1$
        and six congruent copies $T_i$ arranged into the rectangle $R=[0,\sqrt{3}]\times[0,1]$.
        The ``$\pm$'' signs indicate a symmetric or antisymmetric reflection, respectively, in the definition of $\mathcal{P}_a$.
        Assuming zero boundary conditions on $T_1$, the extension by $\mathcal{P}_a$ vanishes on all lines draw in $R$
        and its boundary.
        See \cite{Prager1998}.}
        \label{fig:t1arefl}
    \end{center}
\end{figure}

More succinctly (albeit less precisely) we express this relation as
\begin{equation}\label{eqn:prolong}
    \mathcal{P}_a u(x,y) = \sum_{i=1}^6c_i u(\xi(x_i,y_i),\eta(x_i,y_i)),
\end{equation} where $c_i$ are the appropriate ``$\pm$'' signs.

From \cref{fig:t1arefl} it is clear that the prolongation $\mathcal{P}_au$ vanishes along the sides of the rectangle $R$.
We therefore express $\mathcal{P}_af_a$ as a double-sine series on $R$:
\[
    \sum_{m=1}^\infty\sum_{n=1}^\infty\hat{\mathcal{P}_af_a}(m,n)\sin\left(\frac{\pi mx}{\sqrt{3}}\right)\sin\left(\pi ny\right).
\] The Fourier coefficients of $\mathcal{P}_af_a$ are:
\begin{align}\label{eqn:Pafacoef}
    \hat{\mathcal{P}_af_a}(m,n) &= \frac{4}{\sqrt{3}}\int_{R}\mathcal{P}_af_a(x,y)\sin\left(\frac{\pi mx}{\sqrt{3}}\right)\sin\left(\pi ny\right)\dm{x}\dm{y}\nonumber\\
        &=\frac{4}{\sqrt{3}}\int_{T_1}f_a(\xi,\eta)\left[\sum_{i=1}c_i\sin\left(\frac{\pi mx_i}{\sqrt{3}}\right)\sin\left(\pi ny_i\right)\right]\dm{\xi}\dm{\eta}
        && \text{by } \eqref{eqn:prolong}\nonumber\\
        &= \frac{4}{{\sqrt{3}}}\int_{T_1}f_a(\xi,\eta)u_{m,n}(\xi,\eta)\dm{\xi}\dm{\eta},
\end{align}
where we defined the antisymmetric eigenfunctions as\[
    u_{m,n}(\xi,\eta) := \sum_{i=1}c_i\sin\left(\frac{\pi mx_i}{\sqrt{3}}\right)\sin\left(\pi ny_i\right).
\]
(Recall that each pair $(x_i,y_i)\in R$ is a function of $(\xi,\eta)\in T_1$.)

By explicit computation using the parametrisations of the $T_i$, Práger obtained the following.
\begin{lemma}\label{lemma:umnmult}
    For integers $0<m<n$ with the same parity (\ie both odd or both even), \begin{align}\label{eqn:umn}
        u_{m,n}(x,y) &= 2\sin\left(\frac{\pi m x}{\sqrt{3}}\right)\sin\left(\pi n y\right)\nonumber\\
        &\quad-2(-1)^{(m+n)/2}\sin\left(\frac{\pi x}{2\sqrt{3}}(m+3n)\right)\sin\left(\frac{\pi y}{2}(m-n)\right)\\
        &\quad+2(-1)^{(m-n)/2}\sin\left(\frac{\pi x}{2\sqrt{3}}(m-3n)\right)\sin\left(\frac{\pi y}{2}(m+n)\right)\nonumber.
    \end{align}
    Furthermore, $u_{m,n}\equiv 0$ whenever $n$ and $m$ have different parity \emph{or} $n=m$ \emph{or} $m=3n$.
    Finally, with ${0<m<n}$ as above, define the pairs \begin{align*}
        \left\{\begin{aligned}
            m' &= \tfrac{1}{2}(3n-m),\\
            n' &= \tfrac{1}{2}(m+n);
       \end{aligned}\right.\qand
        &\left\{\begin{aligned}
            m'' &= \tfrac{1}{2}(m+3n),\\
            n'' &= \tfrac{1}{2}(m-n).
        \end{aligned}\right.
    \end{align*}
    Then, \begin{equation}\label{eqn:unmmult}
        u_{m',n'} = -(-1)^{(m-n)/2}u_{m,n} \qand u_{m'',n''} = (-1)^{(m+n)/2}u_{m,n}.
    \end{equation}
\end{lemma}

The next technical result is all that stands between us
and the $L^p$ convergence of a ``$u_{m,n}$''-expansion of $f_a$ on $T_1$.
\begin{lemma}
    The functions $\{u_{m,n} : 0 < m < n \text{ and } m \equiv n \mod 2\}$ are pairwise orthogonal.
    Furthermore, \begin{equation}\label{eqn:Pumn}
        \mathcal{P}_a u_{m,n}(x,y) = u_{m,n}(x,y) \qfa (x,y)\in R,
    \end{equation} from which it follows that
    \begin{equation}\label{eqn:umnnorm}
        \norml{u_{m,n}}{2}{T_1}^2 = \frac{\sqrt{3}}{2}.
    \end{equation}
\end{lemma}

\begin{cor}\label{cor:eigenfs}
    The functions $u_{m,n}$ are eigenfunctions of (minus) the Dirichlet Laplacian on $T_1$,
    with eigenvalue $\pi^2(\tfrac{m^2}{3}+n^2)$.
\end{cor}
\begin{proof}
    Since each $\mathcal{P}_a u_{m,n}$ is a finite combination of eigenfunctions on $R$, we have \[
       -\Delta \mathcal{P}_au_{m,n}(x,y) = \pi^2\left(\frac{m^2}{3}+n^2\right)\mathcal{P}_au_{m,n}(x,y) \qfa (x,y)\in R
    \] by direct calculation.
    This equation holds pointwise owing to the smoothness of the functions;
    therefore, it holds when restricted to $T_1$.
    But $\mathcal{P}_au_{m,n}\equiv u_{m,n}$ on $T_1$, whence \[
        -\Delta u_{m,n}(\xi,\eta) = \pi^2\left(\frac{m^2}{3}+n^2\right)u_{m,n}(\xi,\eta) \qfa (\xi,\eta)\in T.
    \] These functions clearly vanish on the boundary of the triangle by construction.
\end{proof}

So far, our discussion is a recap of the work of \cite{Prager1998},
although we have added a more explicit proof of \cref{cor:eigenfs} than can be found there.
We will now use these results to obtain $L^p$ convergence of the eigenfunctions.

\section{$L^p$ Convergence of the Fourier series of $f_a$}\label{sec:assymconv}

\cref{eqn:umnnorm} tells us that the ``Fourier coefficients'' of $f_a$ (with respect to $u_{m,n}$!) are \[
    f_a^{\triangle}(m,n) := \frac{2}{\sqrt{3}}\int_{T_1}f_a(\xi,\eta)u_{m,n}(\xi,\eta)\dm{\xi}\dm{\eta},
\] and it follows from \eqref{eqn:Pafacoef} that \begin{equation}\label{eqn:Pafacoefeq}
    \hat{\mathcal{P}_af_a}(m,n) = 2f_a^{\triangle}(m,n).
\end{equation}

To help keep track of our indices, we will introduce the sets \[
    \mathcal{U}_{N} := \left\{(m,n) : 0 < m < n \leq N \text{ and } m \equiv n \mod 2\right\}.
\]

Denote the $N$-th ``component-wise'' partial sums on the triangle $T_1$ and rectangle $R$, respectively, by
\begin{align*}
    S^{T_1}_Nf_a(\xi,\eta) &:= \sum_{(m,n)\in\mathcal{U}_N}f_a^{\triangle}(m,n)u_{m,n}(\xi,\eta),\\
    S^R_N\mathcal{P}_af_a(x,y) &:= \sum_{0<m,n\leq N}\hat{\mathcal{P}_af_a}(m,n)\sin\left(\frac{\pi m}{\sqrt{3}}x\right)\sin\left(\pi ny\right).
\end{align*}

The next theorem is a key ingredient in the proof of our main result.
\begin{theorem}\label{thm:Lpfa}
    Let $f\in L^p(T)$ with $1<p<\infty$, and denote its antisymmetric part by $f_a$.
    Then $f_a\in L^p(T_1)$ and $S^{T_1}_Nf_a\to f_a$ in $L^p(T_1)$.
\end{theorem}
\begin{proof}
    Since $f\in L^p(T)$ and $f_a$ is a linear combination of $f$ and its reflections, clearly $f_a\in L^p(T)$
    and, by symmetry (and abuse of notation), ${f_a = f_a|_{T_1}\in L^p(T_1)}$.

    Next, we break up the set of indices with $0<m,n\leq N$ into three sets: \[
        0 < m < n, \quad n < m < 3n \qand 3n < m,
    \] noting that the coefficients when $m=n$ and $m=3n$ vanish.
    From \cref{lemma:umnmult}, we may label these groups as $(m,n)$, $(m',n')$ and $(m'',n'')$,
    respectively.
    Note that such labelling divides $[1,N]\times[1,N]\cap\Z^2$ into three disjoint regions,
    which we label as $I=\mathcal{U}_N$, $II$ and $III$; see \cref{fig:mnregion}.
    
    It follows from \eqref{eqn:unmmult} that \[
        \hat{\mathcal{P}_af_a}(m,n) = -(-1)^{(m-n)/2}\hat{\mathcal{P}_af_a}(m',n')
        = (-1)^{(m+n)/2}\hat{\mathcal{P}_af_a}(m'',n'').
    \]


    Hence, \begin{align*}
        S^R_N\mathcal{P}_af_a(x,y) &= \sum_{(m,n)\in\mathcal{U}_N}\hat{\mathcal{P}_af_a}(m,n)\sin\left(\frac{\pi m x}{\sqrt{3}}\right)\sin\left(\pi ny\right)\\
        &\qquad {} + \sum_{(m',n')\in II}\hat{\mathcal{P}_af_a}(m',n')\sin\left(\frac{\pi m' x}{\sqrt{3}}\right)\sin\left(\pi n'y\right)\\
        &\qquad {} + \sum_{(m'',n'')\in III}\hat{\mathcal{P}_af_a}(m'',n'')\sin\left(\frac{\pi m'' x}{\sqrt{3}}\right)\sin\left(\pi n''y\right)\\
        &=\sum_{(m,n)\in\mathcal{U}_N}\hat{\mathcal{P}_af_a}(m,n)\biggl[ \sin\left(\frac{\pi m x}{\sqrt{3}}\right)\sin\left(\pi ny\right)\\
        &\qquad {} -(-1)^{(m+n)/2}\sin\left(\frac{\pi x}{2\sqrt{3}(m+3n)}\right)\sin\left(\frac{\pi y}{2} (m-n)\right)\\
        &\qquad {} +(-1)^{(m-n)/2}\sin\left(\frac{\pi x}{2\sqrt{3}}(m-3n)\right)\sin\left(\frac{\pi y}{2}(m+n)\right)\biggr]\\
        & = \sum_{(m,n)\in\mathcal{U}_N}\frac{\hat{\mathcal{P}_af_a}(m,n)}{2}u_{m,n}(x,y) \quad \text{by \eqref{eqn:umn}}\\
        & = \sum_{(m,n)\in\mathcal{U}_N}f_a^{\triangle}(m,n)\mathcal{P}_au_{m,n}(x,y).
    \end{align*} The last line follows from \eqref{eqn:Pumn} and \eqref{eqn:Pafacoefeq}.
    Since $\mathcal{P}_a$ is a linear operator, we have therefore proved that \[
        \mathcal{P}_a\left[S^{T_1}_Nf_a\right](x,y) = S^R_N\mathcal{P}_af_a(x,y).
    \] Furthermore, since each transformation $T_1\to T_i$ is an isometry, we have \[
        \norml{\mathcal{P}_au}{p}{R} = 6^{1/p}\norml{u}{p}{T_1} \qfa u\in L^p(T_1).
    \]
    Combining the last two displayed equations with the $L^p$ convergence of double-sine series on $R$,
    we obtain \begin{align*}
        \norml{S_N^{T_1}f_a-f_a}{p}{T_1}^p &= \frac{1}{6}\norml{\mathcal{P}_a\left[S_N^{T_1}f_a\right]-\mathcal{P}_af_a}{p}{R}^p\\
        &= \frac{1}{6}\norml{S_N^{R}\mathcal{P}_af_a-\mathcal{P}_af_a}{p}{R}^p \to 0
    \end{align*} as $N\to\infty$.
\end{proof}

\begin{figure}[h]
    \begin{center}
        \includegraphics{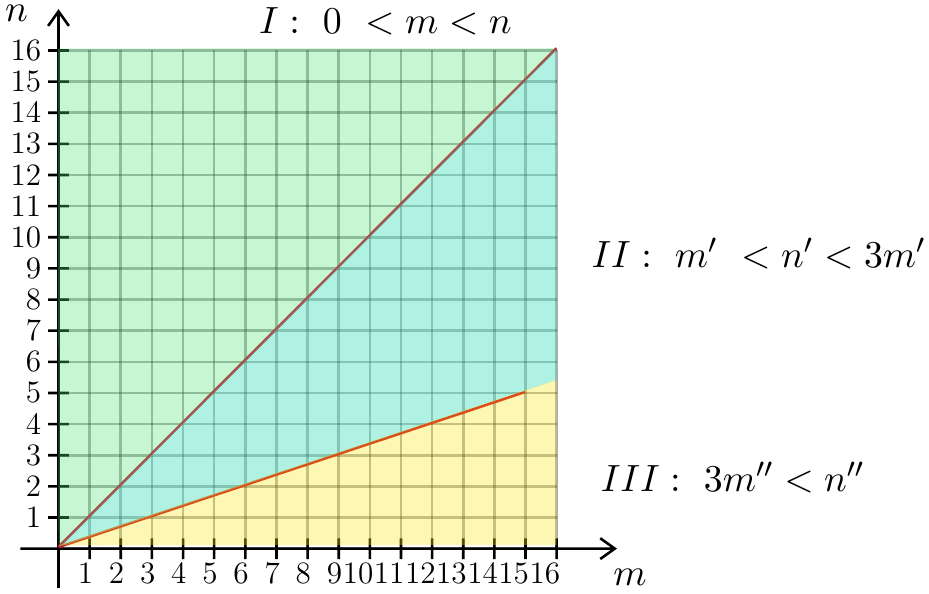}
        \caption{The three regions in the lattice $[1,16]\times[1,16]$.
        The lines correspond to the ``degenerate cases'' $m=n$ and $m=3n$
        in which $u_{m,n}\equiv 0$, and thus the Fourier coefficients vanish.
        Note that all points where $m$ and $n$ have opposite parity will also be excluded.}
        \label{fig:mnregion}
    \end{center}
\end{figure}
\section{Symmetric eigenfunctions on the equilateral triangle}\label{sec:symm}
It is now time to deal with the symmetric part: $f_s$.
To do so, we introduce a \emph{symmetric prolongation} of $u$ to $R$:
\begin{equation}\label{eqn_prolongs}
    \mathcal{P}_su(x,y) := \sum_{i=1}^6d_iu(\xi(x_i),\eta(y_i)) \qtext{where}
    d_i := \begin{cases}
        1 & \text{if } i=1,4,5;\\
        -1 & \text{if } i=2,3,6.
    \end{cases}
\end{equation} This notation uses the same ``shorthand'' as in \eqref{eqn:prolong},
but uses a different combination of reflections and anti-reflections; see \cref{fig:t1srefl}.

\begin{figure}[h]
    \begin{center}
        \includegraphics{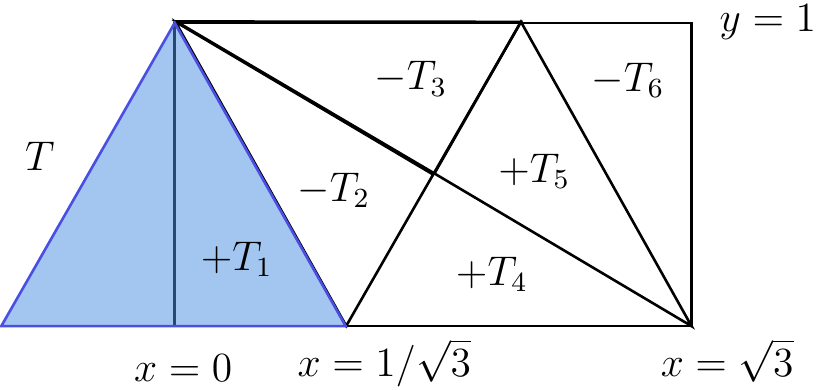}
        \caption{The sign arrangements for the symmetric prolongation $\mathcal{P}_s$.
        See \cite{Prager1998}.}
        \label{fig:t1srefl}
    \end{center}
\end{figure}

From \cref{fig:t1srefl} it is clear that the prolongation $\mathcal{P}_su$ only vanishes on the horizontal sides $y=0$ and $y=1$ of the rectangle,
since for an $x$-symmetric function we cannot assume that $u(x,0)\equiv 0$.
We therefore express $\mathcal{P}_sf_s$ as a cosine-sine series on $R$: \[
    \sum_{m=0}^\infty\sum_{n=1}^\infty \hat{\mathcal{P}_sf_s}(m,n)\cos\left(\frac{\pi mx}{\sqrt{3}}\right)\sin\left(\pi ny\right).
\] When $m,n>0$, the Fourier coefficients of $\hat{\mathcal{P}_sf_s}$ are given by 
\begin{align}\label{eqn:Psfscoef}
    \hat{\mathcal{P}_sf_s}(m,n) &= \frac{4}{\sqrt{3}}\int_{R}\mathcal{P}_sf_s(x,y)\cos\left(\frac{\pi mx}{\sqrt{3}}\right)\sin\left(\pi ny\right)\dm{x}\dm{y}\nonumber\\
        &=\frac{4}{\sqrt{3}}\int_{T_1}f_s(\xi,\eta)\left[\sum_{i=1}d_i\cos\left(\frac{\pi mx_i}{\sqrt{3}}\right)\sin\left(\pi ny_i\right)\right]\dm{\xi}\dm{\eta}
        && \text{by } \eqref{eqn:prolong}\nonumber\\
        &= \frac{4}{{\sqrt{3}}}\int_{T_1}f_s(\xi,\eta)v_{m,n}(\xi,\eta)\dm{\xi}\dm{\eta};
\end{align}
when $m=0$ and $n>0$, they are\begin{equation}\label{eqn:Psfscoef0}
    \hat{\mathcal{P}_sf_s}(0,n) = \frac{2}{\sqrt{3}}\int_{T_1}f_s(\xi,\eta)v_{0,n}(\xi,\eta)\dm{\xi}\dm{\eta}.
\end{equation} Here we have defined the symmetric eigenfunctions as \[
    v_{m,n}(\xi,\eta) := \sum_{i=1}d_i\cos\left(\frac{\pi mx_i}{\sqrt{3}}\right)\sin\left(\pi ny_i\right).
\]
(Recall that each pair $(x_i,y_i)\in R$ is a function of $(\xi,\eta)\in T_1$.)

By explicit computation, Pragér obtained the following.
\begin{lemma}\label{lemma:vmnmult}
    For integers $0\leq m<n$ with the same parity, \begin{align}\label{eqn:vmn}
        v_{m,n}(x,y) &= 2\cos\left(\frac{\pi m x}{\sqrt{3}}\right)\sin\left(\pi n y\right)\nonumber\\
        &\quad+2(-1)^{(m+n)/2}\cos\left(\frac{\pi x}{2\sqrt{3}}(m+3n)\right)\sin\left(\frac{\pi y}{2}(m-n)\right)\\
        &\quad-2(-1)^{(m-n)/2}\cos\left(\frac{\pi x}{2\sqrt{3}}(3n-m)\right)\sin\left(\frac{\pi y}{2}(m+n)\right)\nonumber.
    \end{align}
    Furthermore, $v_{m,n}\equiv 0$ whenever $n$ and $m$ have different parity.
    Finally, with $0<m<n$ as above, define the pairs \begin{align*}
        \left\{\begin{aligned}
            m' &= \tfrac{1}{2}(3n-m),\\
            n' &= \tfrac{1}{2}(m+n);
       \end{aligned}\right.\qand
        &\left\{\begin{aligned}
            m'' &= \tfrac{1}{2}(m+3n),\\
            n'' &= \tfrac{1}{2}(m-n).
        \end{aligned}\right.
    \end{align*}
    Then, \begin{equation}\label{eqn:vnmmult}
        v_{m',n'} = -(-1)^{(m-n)/2}v_{m,n}, \quad v_{m'',n''} = -(-1)^{(m+n)/2}v_{m,n},
    \end{equation} and, for all $n>0$, \begin{equation}\label{eqn:v0nmult}
        v_{0,n} = -(-1)^{n/2}v_{3n/2,n/2}.
    \end{equation}
\end{lemma}
\begin{remark}\label{rem}
Note that we no longer exclude the cases $m=0$ and $m=3n$, since the corresponding $v_{m,n}$ do not vanish,
and must therefore be carefully accounted for when we compute the partial sums.
Note also that $m$ and $n$ must have the same parity, so that $n$ is even when $m=0$,
and, consequently, the right-hand side of \eqref{eqn:v0nmult} is well-defined.
We can thus regard $(m',n') = (3n/2,n/2)$ when $m=0$.
\end{remark}

It is also clear that $v_{m,n}$ are the symmetric eigenfunctions of the Dirichlet Laplacian on $T$,
restricted to $T_1$; the proof is identical to that of \cref{cor:eigenfs}.
\citep[See][for a fuller treatment.]{McCartin2003}

\begin{lemma}
    The eigenfunctions $\{v_{m,n} : 0 \leq m < n, n>0, m \equiv n \mod 2\}$ are pairwise orthogonal.
    Furthermore, \begin{equation}\label{eqn:Pvmn}
        \mathcal{P}_s v_{m,n}(x,y) = v_{m,n}(x,y) \qfa (x,y)\in R,
    \end{equation} from which it follows that
    \begin{align}
        \norml{v_{m,n}}{2}{T_1}^2 &= \frac{\sqrt{3}}{2} &\qfa n,m>0\label{eqn:vmnnorm};\\
        \norml{v_{0,n}}{2}{T_1}^2 &= \sqrt{3} &\qfa n>0\label{eqn:v0nnorm}.
    \end{align}
\end{lemma}

Once again, these results are containted in \cite{Prager1998}.
They are crucial to our proof of $L^p$ convergence in the next section.

\section{$L^p$ Convergence of the Fourier Series of $f_s$}\label{sec:symmconv}

Writing\[
    f_s^{\triangle}(m,n) := \frac{1}{\norml{v_{m,n}}{2}{T_1}^2}\int_{T_1}f_sv_{m,n},  
\]
it immediately follows from \eqref{eqn:Psfscoef},\eqref{eqn:vmnnorm} and \eqref{eqn:Psfscoef0},\eqref{eqn:v0nnorm}, that
\begin{equation}\label{eqn:Psfscoefeq}
    \hat{\mathcal{P}_sf_s}(m,n) = 2f_s^{\triangle}(m,n) \qfa m\geq0,n>0.
\end{equation}

The next theorem is the other key ingredient of our main result.
\begin{theorem}\label{thm:Lpfs}
    Let $f\in L^p(T)$ with $1<p<\infty$, and denote its symmetric part by $f_s$.
    Then $f_s\in L^p(T_1)$ and $S^{T_1}_Nf_s\to f_s$ in $L^p(T_1)$.
\end{theorem}
\begin{proof}
    The proof is essentially the same as that of \cref{thm:Lpfa},
    the only modifications arising from the ``degenerate'' indices $(0,n)$ and $(3n,n)$.

    As before, we break up the set of indices with $0\leq m\leq N$, $0<n\leq N$ into three sets: \[
        0 < m < n, \quad n < m < 3n \qand 3n < m
    \] sparing the cases $m=0$ and $m=3n$.
    From \cref{lemma:vmnmult}, we may label these groups as $(m,n)$, $(m',n')$ and $(m'',n'')$,
    respectively, dividing the lattice ${[0,N]\times[1,N]\cap\Z^2}$ into three disjoint regions;
    see once again \cref{fig:mnregion}.
    
    It follows from \eqref{eqn:vnmmult} that, for all positive indices, \[
        \hat{\mathcal{P}_sf_s}(m,n) = -(-1)^{(m-n)/2}\hat{\mathcal{P}_sf_s}(m',n')
        = (-1)^{(m+n)/2}\hat{\mathcal{P}_sf_s}(m'',n''),
    \] and from \eqref{eqn:v0nmult} that \begin{equation}\label{eqn:hatPsfsmult}
        \hat{\mathcal{P}_sf_s}(0,n) = -(-1)^{n/2}\hat{\mathcal{P}_sf_s}(3n/2,n/2) \qfa n > 0.
    \end{equation}Recall that we identify $(0',n')=(3n/2,n/2)$ and that this is precisely the case $m'=3n'$;
    see \cref{rem}.
    
    We begin by grouping the partial sums on $R$ as follows: \begin{align}
        S_{N}^{R}\mathcal{P}_sf_s(x,y) &= \sum_{\substack{0<m,n\leq N\\m\neq n, m\neq 3n}}
            \hat{\mathcal{P}_sf_s}(m,n)\cos\left(\frac{\pi mx}{\sqrt{3}}\right)\sin\left(\pi ny\right)\label{eqn:Snspos}\\
            &\quad {} +\biggl[\sum_{(0,n)}\hat{\mathcal{P}_sf_s}(0,n)\sin(\pi ny)\nonumber\\
            &\quad {} +\sum_{(0',n')}\hat{\mathcal{P}_sf_s}(0',n')\cos\left(\frac{0'\pi x}{\sqrt{3}}\right)\sin\left(\pi n'y\right)\biggr]
            \label{eqn:Sns0}.
    \end{align}
    
    The calculations for the term \eqref{eqn:Snspos} proceed in the same way as for
    the antisymmetric part $S_N^R\mathcal{P}_af_a$,
    using \eqref{eqn:vmn} instead of \eqref{eqn:umn} after grouping the various cosine-sine terms.
    We will therefore only consider the bracketed term \eqref{eqn:Sns0} in detail.

    The sums in \eqref{eqn:Sns0} are taken over even $n\leq N$ and pairs ${(0',n')=(3n/2,n/2)}$
    corresponding to $m=3n$.
    Using \eqref{eqn:hatPsfsmult} we have
    \begin{align*}
        \sum_{(0,n)}&\hat{\mathcal{P}_sf_s}(0,n)\sin(\pi ny)\nonumber
            +\sum_{(0',n')}\hat{\mathcal{P}_sf_s}(0',n')\cos\left(\frac{0'\pi x}{\sqrt{3}}\right)\sin\left(\pi n'y\right)\\
        &= \sum_{(0,n)}\hat{\mathcal{P}_sf_s}(0,n)\left[\sin\left(\pi n y\right)
            - \mathbf{2}(-1)^{n/2}\cos\left(\frac{2n\pi x}{2\sqrt{3}}\right)\sin\left(\frac{3\pi y}{2}\right)\right]\\
        &= \sum_{(0,n)}\hat{\mathcal{P}_sf_s}(0,n)\biggl[\cos\left(\frac{\pi \cdot 0\cdot x}{\sqrt{3}}\right)\sin\left(\pi n y\right)\\
            &\qquad {} +(-1)^{(0+n)/2}\cos\left(\frac{\pi x}{2\sqrt{3}}(0+3n)\right)\sin\left(\frac{\pi m}{2}(0-n)\right)\\
            &\qquad {} -(-1)^{(0-n)/2}\cos\left(\frac{\pi x}{2\sqrt{3}}(3n-0)\right)\sin\left(\frac{\pi m}{2}(n-0)\right)\biggr]\\
        &= \sum_{(0,n)}\frac{\hat{\mathcal{P}_sf_s}(0,n)}{2}v_{0,n}(x,y)  \quad\text{by \eqref{eqn:vmn}}\\
        &= \sum_{(0,n)}f_s^{\triangle}(0,n)\mathcal{P}_sv_{0,n}(x,y).
    \end{align*}
    The bold factor of $2$ in the second line arises from the fact that the normalisation constant
    for $\hat{\mathcal{P}_sf_s}(m,n)$, $m\neq 0$, is twice that of $\hat{\mathcal{P}_sf_s}(0,n)$;
    see \eqref{eqn:Psfscoef} and \eqref{eqn:Psfscoef0}.
    The last line follows from \eqref{eqn:Pvmn} and \eqref{eqn:Psfscoefeq}.
    Once again we have \[
        \mathcal{P}_s\left[S^{T_1}_Nf_s\right](x,y) = S^R_N\mathcal{P}_sf_s(x,y),
    \] and we can conclude that $S_N^{T_1}f_s\to f_s$ in $L^p(T_1)$.
\end{proof}

\section{Main Results for the Equilateral Triangle}\label{sec:results}

We now wrap up with the main results as they apply to the original equilateral triangle, $T$.

\begin{theorem}[{\citealp{Prager1998}}]
    The functions \begin{align*}
        &u_{m,n}: \quad m,n = 1,2,\ldots, \quad &&m \equiv n\mod 2, \quad 0<m<n;\\
        &v_{m,n}: \quad m=0,1,2,\ldots, n = 1, 2, \ldots, \quad &&m \equiv n \mod 2, \quad 0 \leq m < n,
    \end{align*} form a complete, orthogonal system on $T$, consisting of eigenfunctions of the Dirichlet Laplacian.
    Furthermore, each eigenfunction, $u_{m,n}$ or $v_{m,n}$, corresponds to the eigenvalue $\pi^2\left(\tfrac{m^2}{3}+n^2\right)$.
\end{theorem}

As for the $L^p$ theory, we can combine our \cref{thm:Lpfa,thm:Lpfs} to obtain the full result on $L^p(T)$.
As before, we use the following notation to keep track of our indices in the asymmetric and symmetric parts of the sums:
\begin{align*}
    \mathcal{U}_N &:= \left\{(m,n) : 0 < m < n \leq N \text{ and } m \equiv n \mod 2\right\},\\
    \mathcal{V}_N &:= \left\{(m,n) : 0\leq m < n \text{ and } m \equiv n \mod 2\right\}.
\end{align*}
\begin{theorem}
    Let $f\in L^p(T)$ with $1<p<\infty$. With the notation of the previous sections, let \[
        S_N^{T}f := \sum_{(m,n)\in\mathcal{U}_N}f_a^{\triangle}(m,n)u_{m,n}
            + \sum_{(m,n)\in\mathcal{V}_N}f_s^{\triangle}(m,n)v_{m,n},
    \] Then, $S_N^Tf\to f$ in $L^p(T)$.
\end{theorem}
\begin{proof}
    For any function $f$ on $T$, write its decomposition into a symmetric and antisymmetric part with respect to the $y$-axis: \[
        f = f_a + f_s.
    \] Then clearly \[
        \norml{f_\nu}{p}{T}^p = 2\norml{f_\nu}{p}{T_1}^p \qtext{for} \nu = a,s,
    \] whence \[
        \norml{f}{p}{T} \leq 2^{1/p}\left(\norml{f_a}{p}{T_1}+\norml{f_s}{p}{T_1}\right).
    \] Hence, since \[
        \left[S_N^Tf\right]_{\nu} = S_{N}^{T_1}f_{\nu} \qtext{for} \nu = a,s,
    \] it follows that \[
        \norml{S_{N}^Tf-f}{p}{T} \leq 2^{1/p}\left(\norml{S_N^{T_1}f_a-f_a}{p}{T_1}
            +\norml{S_N^{T_1}f_s-f_s}{p}{T_1}\right) \to 0
    \] by \cref{thm:Lpfa,thm:Lpfs}.
\end{proof}
\section{Concluding remarks}\label{sec:conc}

Our argument made use of Pr\'ager's cunning triangle-to-rectangle transformation
in order to reduce the convergence problem on the triangle to the well-known convergence on the rectangle.
In a similar, but more direct way, we were able to obtain these results for the 45-90-45 triangle.

There are, however, limitations to this approach, as a consequence of the following theorem due to \cite{Lame1833}
and reported as Theorem 3.1 in \cite{McCartin2003}.
\begin{theorem}[Lam\'e's Fundamental Theorem]
    Suppose that $f(x,y)$ can be represented by the trigonometric series \[
        f(x,y) = \sum_{i}A_i\sin(\lambda_ix+\mu_iy+\alpha_i) + B_i\cos(\lambda_ix+\mu_iy+\beta_i)
    \] with $\lambda_i^2+\mu_i^2=k^2$.
    Then $f$ is antisymmetric about any line about which it vanishes.
\end{theorem}

\begin{figure}[ht]
    \captionsetup[subfigure]{font=footnotesize}
    \centering
    \subcaptionbox{In general, the diagonal is not a line of\\symmetry, so it cannot be a nodal line.}[.5\textwidth]{%
    \includegraphics{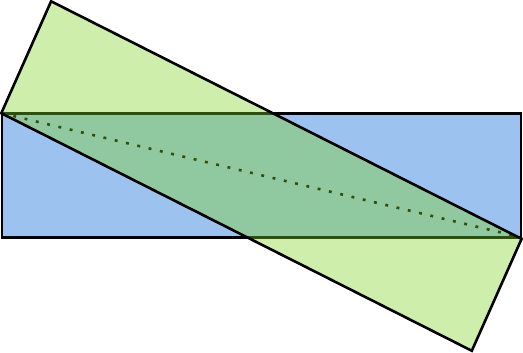}}%
    \subcaptionbox{Mimicking the Pr\'ager construction with an angle of $\pi/2n$ ($n\neq2,3$) does not tile a rectangle.}[.5\textwidth]{%
    \includegraphics{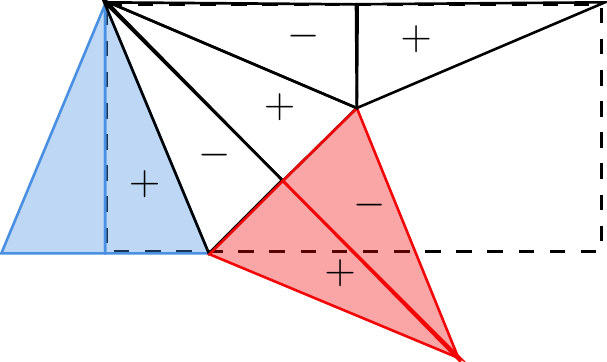}}
    \caption{Limitations of the triangle-to-rectangle constructions.} 
    \label{fig:extensions}
\end{figure}

The implication for our argument is that, assuming Dirichlet boundary conditions on the hypothenuse of $T_1$,
we will generate a nodal (\ie vanishing) line along the diagonal of the rectangle (see, \eg, \cref{fig:t1arefl}).
Lam\'e's Theorem then requires the eigenfunctions to have a line of anti-symmetry along the diagonal.
The only possibilities are symmetry along the diagonal line itself, as in the square (see \cref{fig:extensions},(a));
or an arrangement of smaller triangles inside the rectangle
as in Pr\'{a}ger's construction with three hemiequilateral triangles.
In this case, the upper right-angle of the rectangle is cut into three angles of $\pi/6$.
Attempts to replace this decomposition by $n>3$ triangles with angles $\pi/2n$
will not tile a rectangle (see \cref{fig:extensions},(b)).

A more general treatment of eigenfunction expansions for the Dirichlet Laplacian in arbitrary triangular domains
requires further research.

\section*{Acknowledgments}
I owe a special thanks to my doctoral advisor, Prof James C. Robinson,
for many insightful discussions and his helpful criticism of early drafts.
This work was supported by the EPSRC/EP/V520226/12443915 studentship and the Warwick Mathematics Institute.

\bibliography{bibliography}

\end{document}